\documentclass[10pt,oneside]{amsart}

\usepackage{latexsym,amssymb,amsmath,amscd,amsxtra,amsfonts,amsthm}
\usepackage{fancyhdr,array,dsfont}

\newtheorem{theorem}{Theorem}[section]
\newtheorem{lemma}[theorem]{Lemma}
\newtheorem{proposition}[theorem]{Proposition}
\newtheorem{remark}[theorem]{Remark}

\newcommand{\beq}{\begin{eqnarray*}}
\newcommand{\eeq}{\end{eqnarray*}}
\newcommand{\beqn}{\begin{eqnarray}}
\newcommand{\eeqn}{\end{eqnarray}}

\numberwithin{equation}{section}

\begin{document}





\title[Gaussian Integrability under the Lyapunov Condition]{Gaussian Integrability of Distance Function under the Lyapunov Condition}

\author[Y. LIU]{Yuan LIU}
\address{Yuan LIU, Institute of Applied Mathematics, Academy of Mathematics and Systems Science,
Chinese Academy of Sciences, Beijing 100190, China}
\email{liuyuan@amss.ac.cn}

\date{\today}

\begin{abstract}
In this note we give a direct proof of the Gaussian integrability of distance function as $\mu e^{\delta d^2(x,x_0)} < \infty$ for some $\delta>0$ provided the Lyapunov condition holds for symmetric diffusion Markov operators, which answers a question proposed in Cattiaux-Guillin-Wu \cite[Page 295]{CGW}. The similar argument still works for diffusions processes with unbounded diffusion coefficients and for jump processes such as birth-death chains. An analogous discussion is also made under the Gozlan's condition arising from \cite[Proposition 3.5]{Gozlan}.
\end{abstract}



\maketitle

\allowdisplaybreaks

\section{Introduction}
 \label{Intro}
 \setcounter{equation}{0}
In this note, we will investigate how to directly derive the Gaussian integrability from two kinds of criteria for the Talagrand's inequality $W_2H$, say the Lyapunov condition and Gozlan's condition presented in a symmetric diffusion Markov setting. Referring to Bakry-Gentil-Ledoux \cite{BGL}, in the sequel we denote by $E$ a complete connected  Riemannian manifold of finite dimension, $d$ the geodesic distance, $\mathrm{d}x$ the volume measure, $\mu(\mathrm{d}x) = e^{-V(x)}\mathrm{d}x$ a probability measure with $V\in C^2(E)$, $\mathrm{L}=\Delta - \nabla V\cdot \nabla$ the $\mu$-symmetric diffusion operator, $\Gamma(f,g) = \nabla f\cdot \nabla g$ the carr\'{e} du champ operator, and $\mathcal{E}$ the associated Dirichlet form, which satisfy the formula for integration by parts
   \[ \int \nabla f \cdot \nabla g \mathrm{d}\mu = -\int f \mathrm{L}g  \mathrm{d}\mu, \ \ f\in \mathcal{D}(\mathcal{E}), g\in \mathcal{D}(\mathrm{L}). \]

First of all, say $W\geqslant 1$ is a Lyapunov function if there exist two constants $b\geqslant 0$ and $c>0$ such that for some $x_0\in E$ and any $x\in E$
  \beqn
     \mathrm{L}W \leqslant (-cd^2(x,x_0) + b) W. \label{eqLya}
  \eeqn
More generally, to avoid assuming the integrability and second-order differentiability of $W$, it is convenient to introduce a locally Lipschitz function $U> 0$ such that in the sense of distribution
  \beqn
     \mathrm{L}U + \left| \nabla U \right|^2 \leqslant -cd^2(x,x_0) + b, \label{eqLyaU}
  \eeqn
which means that for any nonnegative $h\in C^\infty_{\textrm{c}}(E)$ holds
  \[ \int \left(\mathrm{L}U + \left| \nabla U \right|^2\right) h \mathrm{d}\mu
      := \int U\mathrm{L}h + \left| \nabla U \right|^2 h\mathrm{d}\mu \leqslant \int \left(-cd^2(x,x_0) + b\right) h \mathrm{d}\mu. \]
When $W\in C^2(E)$, (\ref{eqLya}) and (\ref{eqLyaU}) are equivalent by taking $U=\log W$. And it is not hard to see that (\ref{eqLya}) implies a weaker version for some $c'$, $b'$ and $R$
  \beqn
     \mathrm{L}W \leqslant -c'W + b'\mathbf{1}_{B(0,R)}. \label{eqLyaWeak}
  \eeqn

The Lyapunov condition plays a powerful role in studying coercive functional inequalities or estimating convergence rate of Markov processes, which even works as a substitute of curvature-dimension condition in some cases. Cattiaux-Guillin \cite{CG} gave a comprehensive review on this topic, and here we would like to take partial literature into account. A simple proof of the Poincar\'{e} inequality through (\ref{eqLyaWeak}) can be found in Bakry-Barthe-Cattiaux-Guillin \cite{Bakry-Barthe}. The $L^1$ transport-information inequality $W_1I$ was discussed further under (\ref{eqLya}) by Guillin-L\'{e}onard-Wu-Yao \cite{GLWY}. Then Cattiaux-Guillin-Wu \cite{CGW} showed the Talagrand's inequality and logarithmic Sobolev inequality (LSI for short) provided (\ref{eqLyaU}), which was also applied to weighted LSIs for heavy tailed distributions by \cite{CGW2}. Most recently, Guillin-Joulin \cite{GJ} obtained non-Gaussian concentration estimates by means of functional inequalities with some kind of Lyapunov condition yet.

According to \cite[Lemma 3.5]{CGW}, it was proved that if (\ref{eqLyaU}) holds, there exist some $\delta>0$ and $x_0\in E$ such that
  \beqn
    \int e^{\delta d^2(x,x_0)}\mathrm{d}\mu(x) < \infty, \label{eqGauInt}
  \eeqn
which is necessary to derive $W_2H$. Their proof starts from (\ref{eqLyaU}) and the spectral gap to show $W_1I$ due to \cite{GLWY}. It then follows a $L^1$ transport-entropy inequality $W_1H$ by Guillin-L\'{e}onard-Wang-Wu \cite{GLWW}, which is equivalent to (\ref{eqGauInt}) by Djellout-Guillin-Wu \cite{DGW}. The strategy relies on a series of works on transport inequalities, thereupon the authors of \cite{CGW} feel interested in finding a simple or direct proof of (\ref{eqGauInt}), see \cite[Page 295]{CGW}.

Indeed, there exists an elementary proof, and we actually obtain

\begin{proposition}\label{propGauInt}
If (\ref{eqLyaU}) holds, then $\mu e^{\delta d^2(x,x_0)}<\infty $ for any $\delta < \sqrt{c}$.
\end{proposition}
\begin{remark}
The upper bound for $\delta$ is sharp. For instance, let $\mathrm{d}\mu = \frac{1}{\sqrt{2\pi}}e^{-\frac12 |x|^2}\mathrm{d}x$ and $\mathrm{L} = \frac{\mathrm{d}^2}{\mathrm{d}x^2} - x\frac{\mathrm{d}}{\mathrm{d}x}$ associated to one-dimensional Ornstein-Uhlenbeck process, then $W=e^{\frac14 |x|^2}$ satisfies $\mathrm{L}W \leqslant (-\frac14 |x|^2 + \frac12) W$, which exactly gives $\delta < \sqrt{c} = \frac12$.
\end{remark}
\begin{remark}
A weak version $\mathrm{L}W \leqslant (-cd^p(x,x_0) + b) W$ with $p<2$
is not enough to derive the Gaussian integrability, since $W = \exp\left(\frac12 (1+|x|^2)^q\right)$ with $2(q-1) = p$ fulfills (\ref{eqLyaU}) for $\mathrm{d}\mu = \frac{1}{Z}e^{-(1+|x|^2)^{\frac{p}{2}}}\mathrm{d}x$, where $Z$ is a normalization factor.
\end{remark}

The same argument can be extended to diffusion processes with unbounded diffusion coefficients. Define an infinitesimal generator in $\mathbb{R}^m$
  \[ \mathrm{L}_a = \frac12 \sum\limits_{i,j=1}^m a^{ij}(x) \frac{\partial}{\partial x_i}\frac{\partial}{\partial x_j}
       + \sum\limits_{i=1}^m b^{i}(x) \frac{\partial}{\partial x_i}, \]
where $A = (a^{ij})_{i,j=1}^m$ is symmetric positive-definite and $b^{i} = \frac12 \left(\sum\nolimits_{j=1}^m \frac{\partial a^{ij}}{\partial x_j} - a^{ij} \frac{\partial V}{\partial x_j}\right)$ so that $\mathrm{L}_a$ admits an invariant probability measure $\mathrm{d}\mu(x) = e^{-V}\mathrm{d}x$. Then define the Carr\'{e}du champ operator by means of
  \[ \Gamma_a(f,g) = \frac{1}{2}\left[\; \mathrm{L}_a(fg) - f\mathrm{L}_ag - g \mathrm{L}_af \;\right] =  \frac{1}{2}\langle \nabla f, A\nabla g\rangle , \]
which satisfies the integration by parts formula for $f,g\in C_\mathrm{c}^\infty (\mathbb{R}^m)$
  \[ -\int f\mathrm{L}_ag \mathrm{d}\mu = \int \Gamma_a(f,g)  \mathrm{d}\mu =: \mathcal{E}_a(f,g). \]
Thanks to the Lyapunov type criterion by Stroock-Varadhan \cite[Theorem 10.2.1]{SV}, it can be quickly derived that $\mathrm{L}_a$ corresponds to a non-explosive diffusion process provided that (\ref{eqLya}) holds by substituting $\mathrm{L}$ to $\mathrm{L}_a$
  \beqn
     \mathrm{L}_a W \leqslant (-cd^2(x,x_0) + b) W \label{eqLyaA}
  \eeqn
with $\lim\nolimits_{|x|\to \infty} W = \infty$.

However, if $a^{ij}$ is unbounded, (\ref{eqLyaA}) is not enough to get the Gaussian integrability for $\mu$. Consider one-dimensional case, when $a^{ij} = a(x) = o(|x|^4)$, we take $V = \frac{x^2}{2\sqrt{a}}$ and $W= e^{\delta V}$ for small $\delta > 0$ so that (\ref{eqLyaA}) holds but $V$ has a growth rate slower than quadratic. On the other hand, when $a(x) = O(|x|^4)$ or grows even faster, (\ref{eqLyaA}) is useless to yield a Poincar\'{e} type inequality so that we have no effective calculus on the integrability of $e^{\delta d^2(x,x_0)}$. For this reason, a stronger condition is necessary.

\begin{proposition}\label{propGauIntUnb}
Let $\lambda_{\mathrm{max}}$ be the maximal eigenvalue of $A$ satisfying $\mu\lambda_{\mathrm{max}}<\infty$. Suppose there exists a Lyapunov function $W\geqslant 1$ with two constants $b\geqslant 0$ and $c>0$ such that for some $x_0\in \mathbb{R}^m$ and any $x\in \mathbb{R}^m$
  \beqn
     \mathrm{L}_a W \leqslant  (-cd^2(x,x_0) + b)\lambda_{\mathrm{max}} W. \label{eqLyaUnb}
  \eeqn
Then $\mu \left( e^{\delta d^2(x,x_0)}\lambda_{\mathrm{max}} \right)<\infty $ for any $\delta < \sqrt{c}$.
\end{proposition}
\begin{remark}
(\ref{eqLyaUnb}) is natural, for instance, if there exist $V$ and $W$ satisfying (\ref{eqLya}) over $\mathbb{R}$, then (\ref{eqLyaUnb}) follows automatically provided that $\lim\limits_{|x| \to \infty} \frac{a'W'}{aW|x|^2} = 0$. Moreover, there is no need to assume $\lambda_{\mathrm{max}}\geqslant \lambda >0$ uniformly on $\mathbb{R}^m$.
\end{remark}

Another possible extension is about jump processes (see Bass \cite{Bass}). To clarify the effect from jumps part, we simply consider the infinitesimal generator of the form
  \[ \mathrm{L}_{\nu} = \int_{\mathbb{R}^m-\{0\}} \left[ f(x+y)-f(x) - \nabla f \cdot y \mathbf{1}_{0<|y|<1}(y) \right] \nu(x,\mathrm{d}y), \]
Where $\nu$ satisfies $\int_{\mathbb{R}^m-\{0\}} \min\{1, |y|^2\} \nu(x,\mathrm{d}y) < \infty$. Suppose that $\mathrm{L}_{\nu}$ admits an invariant probability measure $\mu$, and the Carr\'{e}du champ operator
  \beq
    \Gamma_{\nu}(f,g) &=& \frac{1}{2}\left[\; \mathrm{L}_{\nu}(fg) - f\mathrm{L}_{\nu} g - g \mathrm{L}_{\nu} f \;\right] \\
      &=& \frac12 \int_{\mathbb{R}^m-\{0\}} \left[f(x+y)-f(x)\right]\left[g(x+y)-g(x)\right] \nu(x,\mathrm{d}y)
  \eeq
fulfills the integration by parts formula for $f,g \in C_\mathrm{c}^\infty(\mathbb{R}^m)$
  \[ -\int f\mathrm{L}_{\nu} g \mathrm{d}\mu = \int \Gamma_{\nu}(f,g)  \mathrm{d}\mu =: \mathcal{E}_{\nu}(f,g). \]
Then define an intrinsic (pseudo)metric according to Sturm \cite[Definition 6.5]{Sturm}
  \[ \rho(x,y) := \sup\{f(x)-f(y): \Gamma_{\nu}(f,f) \leqslant 1\}, \]
which gives $\Gamma_{\nu}(\rho(x,x_0),\rho(x,x_0)) \leqslant 1$ if $\rho(x,x_0) \in \mathcal{D}(\mathcal{E}_{\nu})$. For convenience, we also require that $\lim\nolimits_{|x|\to\infty} \rho(x,x_0) = \infty$.

The setting includes discrete Markov chains. For example, consider a birth-death process on $\mathbb{N}$ with strictly positive birth rates $b_i$ and death rates $d_i$ except $d_0=0$. Let $r_0 = 1$ and $r_i = \frac{b_0b_1\cdots b_{i-1}}{d_1d_2\cdots d_{i}}$ for $i\geqslant 1$, we can take $\nu(i,y) = b_i\delta_{1}(y) + d_i \delta_{-1}(y)$ and $\mu(i) = \frac{r_i}{r_0+r_1+\cdots}$ provided the series converges, and then $\mathcal{E}_\nu$ has an alternative expression $\mathcal{E}_\nu(f,g) = \sum_{i=0}^\infty [f(i+1)-f(i)][g(i+1)-g(i)]b_i\mu_i$, which determines the intrinsic metric $\rho(i,j) = b_i^{-\frac{1}{2}} + b_{i+1}^{-\frac{1}{2}} + \cdots b_{j-1}^{-\frac{1}{2}}$ for $i\leqslant j$.

\begin{proposition}\label{propGauIntJump}
Suppose there exist some $x_0\in \mathbb{R}^m$ and a constant $K>0$ such that for all $x\in \mathbb{R}^m$ and all $y\in \mathrm{Supp}\nu$
  \beqn
     \left| \rho^2(x+y,x_0) - \rho^2(x,x_0) \right| \leqslant K.  \label{eqJumpMetric}
  \eeqn
Suppose also there exists a Lyapunov function $W\geqslant 1$ with two constants $b\geqslant 0$ and $c>0$ such that for any $x\in \mathbb{R}^m$
  \beqn
     \mathrm{L}_\nu W \leqslant  (-c\rho^2(x,x_0) + b) W. \label{eqLyaJump}
  \eeqn
Then $\mu e^{\delta \rho^2(x,x_0)} <\infty $ for $\delta< C\min\{\sqrt{c},K^{-1}\}$ with some multiple $C\in (0,1]$.
\end{proposition}
\begin{remark}
For a birth-death process referring to Cattiaux-Guillin-Wang-Wu \cite{CGWW}, let $b_i=d_i=i^a\log^\alpha(i+1)$ with $a\geqslant 2$ and $\alpha\in \mathbb{R}$ except $b_0 =1$, let $W=1+i^\gamma$ with $0<\gamma<1$, then $\mu(i) \asymp b_i^{-1}$, $\mathrm{L}_\nu W \leqslant -ci^{a-2}\log^\alpha(i+1) W$. Take $a=2, \alpha =1, \gamma = \frac12$, it follows $\rho(i,0) \asymp \log^{\frac12}(i+1)$ satisfying (\ref{eqJumpMetric}-\ref{eqLyaJump}) and then $\mu e^{\delta \rho^2} <\infty $ for $\delta < 1$. If $a=2, \alpha <1, \gamma = \frac12$, (\ref{eqJumpMetric}) holds, but (\ref{eqLyaJump}) fails and so does the Gaussian integrability; if $b_i = (i+1)^{\frac12}$ and $d_i = ib_i$, then $\mu(i) \asymp (i!b_i)^{-1}$, $\rho(i,0) \asymp i^{\frac34}$ and (\ref{eqLyaJump}) holds for $W=2^i$, but (\ref{eqJumpMetric}) fails and so does the Gaussian integrability again.
\end{remark}

We further investigate another criterion for transport-entropy inequalities. According to Gozlan \cite[Proposition 3.5]{Gozlan}, let $\mu$ be a probability on $R^m$, suppose there exists $\omega\in C^{3}(\mathbb{R})$ with $\omega'(0)> 0$, $\left|\frac{\omega^{(3)}}{\omega'^3}\right| \leqslant M$ for some constant $M$, and
  \beqn
    \liminf\limits_{|x|\to \infty}\; \frac{1}{u^2} \sum\limits_{i=1}^m \left[ \frac{1}{10} \left(\frac{\partial V}{\partial x_i}\right)^2\left(\frac{x}{u}\right) - \frac{\partial^2 V}{\partial x_i^2}\left(\frac{x}{u}\right) \right] \frac{1}{\omega'(x_i)^2} > mM \label{eqGoz}
  \eeqn
for some constant $u>0$, then a transport-entropy inequality holds with the cost function $d_\omega(x,y) = \left(\sum\limits_{i=1}^m \left|\omega(x_i) -\omega(y_i)\right|^2\right)^{\frac12}$. An interesting case is to set
  \[ \omega(t) = \int_0^t \sqrt{1+s^2}\mathrm{d}s = \frac{t}{2}\sqrt{1+t^2} + \frac12\log\left| t+\sqrt{1+t^2} \right| \]
satisfying $\omega'(0) = 1$ and $\left|\frac{\omega^{(3)}}{\omega'^3}(t)\right| = (1+t^2)^{-3}\leqslant 1$, which corresponds to $W_2H$.

In \cite{CGW}, it was pointed out that (\ref{eqGoz}) is not comparable to the Lyapunov condition (\ref{eqLyaU}) in general. Using the similar argument, we still have
\begin{proposition}\label{propGauIntGoz}
If the Gozlan's type condition holds, i.e.
   \beqn
      \liminf\limits_{|x|\to \infty}\; \sum\limits_{i=1}^m \left[ \frac{23}{27} \left(\frac{\partial V}{\partial x_i}\right)^2(x) - \frac{\partial^2 V}{\partial x_i^2}(x) \right] \frac{1}{1+x_i^2} \geqslant m, \label{eqGozType}
   \eeqn
then $\mu e^{\delta |x|^2}<\infty$ for any $\delta < \frac{2(\sqrt{m} - \sqrt{m-1})}{3\sqrt{3m}}$.
\end{proposition}
\begin{remark}
To yield the Gaussian integrability, or equivalently $W_1H$, the original constant $\frac{1}{10}$ in (\ref{eqGoz}) can be increased to arbitrary $a<1-\frac{4}{27}\frac{m-1}{m}$. So it is convenient to take $a=\frac{23}{27}$. Except $m=1$, it is unlikely to allow $a$ approaching $1$, according to the estimates in Lemma \ref{lemGozPoin} below.
\end{remark}

The next two sections will supply the proofs of all propositions respectively.

\bigskip
\section{Proof of Proposition \ref{propGauInt}, \ref{propGauIntUnb} and \ref{propGauIntJump}}
\label{Proof1}
 \setcounter{equation}{0}

Under the Lyapunov condition (\ref{eqLyaU}), \cite[Lemma 3.4]{CGW} asserts
   \beqn
      \int h^2(x)d^2(x,x_0)\mathrm{d}\mu(x) \leqslant \frac{1}{c} \int |\nabla h|^2 \mathrm{d}\mu + \frac{b}{c}\int h^2 \mathrm{d}\mu, \ \ \forall h\in \mathcal{D}(\mathcal{E}). \label{eqTransLya}
   \eeqn
The technique of proof is the same as in \cite[Page 64]{Bakry-Barthe}.

Now, we prove Proposition \ref{propGauInt}.
\begin{proof}
Let $\beta_n = \int d^{2n}(x,x_0) \mathrm{d}\mu$, which satisfies a recursion by using (\ref{eqTransLya}) that
  \beqn
    \beta_n &=& \int d^{2(n-1)}(x,x_0)d^2(x,x_0) \mathrm{d}\mu \nonumber\\
    &\leqslant& \frac{1}{c} \int |\nabla d^{n-1}(x,x_0)|^2 \mathrm{d}\mu + \frac{b}{c}\beta_{n-1}  = \frac{(n-1)^2}{c}\beta_{n-2} + \frac{b}{c}\beta_{n-1}. \label{eqRecur}
  \eeqn
Since $\beta_0 = 1$ and $\beta_1\leqslant \frac{b}{c}$, we get the integrability of all $d^{2n}(x,x_0)$.

Combining the H\"{o}lder inequality with (\ref{eqRecur}) gives
  \[ \beta_n = \int d^{n+1}(x,x_0)d^{n-1}(x,x_0) \mathrm{d}\mu \leqslant \beta_{n+1}^{\frac12}\beta_{n-1}^{\frac12}
      \leqslant \left(\frac{n^2}{c} \beta_{n-1} + \frac{b}{c}\beta_n \right)^{\frac12} \beta_{n-1}^{\frac12}, \]
which implies
  \[ \beta_n \leqslant \frac{\frac{b}{c} + \sqrt{\frac{b^2}{c^2} + \frac{4n^2}{c}}}{2} \beta_{n-1}
       \leqslant ({\textstyle \frac{b}{c} + \frac{n}{\sqrt{c}}}) \beta_{n-1}. \]
Taking any $\gamma>\frac{1}{\sqrt{c}}$ gives $\frac{b}{c} + \frac{n}{\sqrt{c}} \leqslant \gamma n$ for big $n$, which yields some $C>0$ such that
  \[ \beta_n \leqslant C\gamma^n n!, \ \ \ \forall n\geqslant 1. \]

Hence, for any $\delta < \gamma^{-1} < \sqrt{c}$, we have by the Fatou's lemma
  \beqn
     && \int e^{\delta d^2(x,x_0)}\mathrm{d}\mu \ =\ \int \lim\limits_{k\to \infty} \sum\limits_{n=0}^k \left( \delta d^2(x,x_0) \right)^n/n! \;\mathrm{d}\mu \label{eqExpInt}\\
     &\leqslant& \liminf\limits_{k\to \infty} \int \sum\limits_{n=0}^k \left( \delta d^2(x,x_0) \right)^n/n! \;\mathrm{d}\mu \ =\ \liminf\limits_{k\to \infty} \sum\limits_{n=0}^k  \delta^n\beta_n/n! \ \leqslant\ \frac{C}{1-\delta\gamma}. \nonumber
  \eeqn
The proof is completed.
\end{proof}

The proof of Proposition \ref{propGauIntUnb} is almost the same.
\begin{proof}
Using the Lyapunov condition (\ref{eqLyaUnb}) with the technique from \cite[Page 64]{Bakry-Barthe} gives a similar inequality for $h\in \mathcal{D}(\mathcal{E}_a)$ as (\ref{eqTransLya}) that
  \beq
    \int h^2(x)d^2(x,x_0)\lambda_{\mathrm{max}}\mathrm{d}\mu
    &\leqslant& \frac{1}{c} \int h^2 \cdot \frac{-\mathrm{L}_a W}{W}\mathrm{d}\mu + \frac{b}{c}\int h^2\lambda_{\mathrm{max}} \mathrm{d}\mu \\
    &=& \frac{1}{c} \int \Gamma_a (\frac{h^2}{W}, W) \mathrm{d}\mu + \frac{b}{c}\int h^2\lambda_{\mathrm{max}} \mathrm{d}\mu \\
    &=& \frac{1}{c} \int \Gamma_a(h,h) - W^2\Gamma_a(\frac{h}{W},\frac{h}{W})  \mathrm{d}\mu + \frac{b}{c}\int h^2 \lambda_{\mathrm{max}} \mathrm{d}\mu \\
    &\leqslant& \frac{1}{c} \int |\nabla h|^2 \lambda_{\mathrm{max}} \mathrm{d}\mu + \frac{b}{c}\int h^2 \lambda_{\mathrm{max}} \mathrm{d}\mu.
  \eeq

Let $\beta_n = \int d^{2n}(x,x_0)\lambda_{\mathrm{max}} \mathrm{d}\mu$, which satisfies
  \beq
    \beta_n &=& \int d^{2(n-1)}(x,x_0)d^2(x,x_0)\lambda_{\mathrm{max}} \mathrm{d}\mu \\
    &\leqslant& \frac{1}{c} \int |\nabla d^{n-1}(x,x_0)|^2\lambda_{\mathrm{max}} \mathrm{d}\mu + \frac{b}{c}\beta_{n-1}  = \frac{(n-1)^2}{c}\beta_{n-2} + \frac{b}{c}\beta_{n-1}.
  \eeq
Following rest steps in the previous proof, we get the Gaussian integrability.
\end{proof}

At the end of this section, we prove Proposition \ref{propGauIntJump} by a little different method.
\begin{proof}
The strategy contains three steps.

{\bf Step 1}. Denote $\rho_t(x) = \sqrt{\rho^2(x,x_0) + t}$ with some parameter $t>0$. Using the technique in \cite[Page 64]{Bakry-Barthe} again, we have by Condition (\ref{eqLyaJump}) that for $h\in \mathcal{D}(\mathcal{E}_\nu)$
  \beq
    \int h^2\rho_t^2\mathrm{d}\mu
    &\leqslant& \frac{1}{c} \int h^2 \cdot \frac{-\mathrm{L}_a W}{W}\mathrm{d}\mu + \left(\frac{b}{c}+t\right)\int h^2\mathrm{d}\mu \\
    &=& \frac{1}{c} \int \Gamma_\nu \left(\frac{h^2}{W}, W\right) \mathrm{d}\mu + \frac{b+ct}{c}\int h^2 \mathrm{d}\mu \\
    &=& \frac{1}{c}\cdot \frac12 \int\int_{\mathbb{R}^m-\{0\}} - \left| h(x+y)\frac{W(x)^{\frac12}}{W(x+y)^{\frac12}} - h(x)\frac{W(x+y)^{\frac12}}{W(x)^{\frac12}} \right|^2 \\
    && \ \ \ \ +\ |h(x+y)-h(x)|^2 \nu(x,\mathrm{d}y)\mu(\mathrm{d}x)\ +\ \frac{b+ct}{c} \int h^2 \mathrm{d}\mu \\
    &\leqslant& \frac{1}{c} \int \Gamma_\nu(h,h) \mathrm{d}\mu + \frac{b+ct}{c}\int h^2 \mathrm{d}\mu.
  \eeq

{\bf Step 2}. Basically, our aim is to estimate $\int_\Omega e^{\delta  \rho(x,x_0)^2} \mathrm{d}\mu(x)$ for any bounded domain $\Omega$, while the integration by parts requires to regularize the characteristic function $\mathbf{1}_{\Omega}$. It is usually a routine but with a few tricks in this case.

Define a family of $\phi_r\in C^1(\mathbb{R}^+)$ with any $r>0$ and some constant $N>0$ as
   \[ \phi_r(s) = \left\{
       \begin{array}{ll}
        1, & s\leqslant r;\\
        2(\frac{s-r}{N})^3 - 3(\frac{s-r}{N})^2 + 1, & r<s < r+N;\\
        0, & s\geqslant r+N,
       \end{array}
      \right.  \]
which satisfies $0\leqslant \phi_r \leqslant 1$ and $|\phi_r'|\leqslant \frac{3}{2N} \mathbf{1}_{r<s < r+N}$.

Let $f= e^{\frac{\delta}{2} \rho_t^2}$ and $f_r=\phi_r(\rho_t^2) f$. Let $h_r= \frac{f_r}{\rho_t}$, we have by Step 1
  \beqn
    \int f_r^2 \mathrm{d}\mu
     = \int h_r^2 \rho_t^2 \mathrm{d}\mu \leqslant \frac{1}{c} \int \Gamma_\nu(h_r,h_r) \mathrm{d}\mu + \frac{b+ct}{c}\int h_r^2 \mathrm{d}\mu. \label{eqTrunct}
  \eeqn
For convenience, rewrite $h_r= \phi_r(\rho_t^2)\psi(\rho_t)$ by putting $\psi(s) := \frac{e^{\frac{\delta}{2} s^2}}{s}$.

Take $t=2\delta^{-1}$ so that $\psi$ is increasing on $[\sqrt{t}, \infty)$. Using the mean value theorem respectively to $\psi$ and $\phi_r$ yields that for any $x\in \mathbb{R}^m$ and $y\in \mathrm{Supp}\nu$, there exist $\xi$ and $\zeta$ both falling between $\rho(x+y)$ and $\rho(x)$ such that
  \beq
    &&|h_r(x+y)-h_r(x)| \\
    &\leqslant& \phi_r(\rho_t^2(x)) \cdot|\psi(\rho_t(x+y)) - \psi(\rho_t(x))|\\
    && \ \ \ \ + \ \psi(\rho_t(x+y)) \cdot|\phi_r(\rho_t^2(x+y)) - \phi_r(\rho_t^2(x))| \\
    &=& \phi_r(\rho_t^2(x)) \cdot \left|\delta - \xi^{-2}\right| e^{\frac{\delta}{2} \xi^2}  \cdot |\rho_t(x+y) - \rho_t(x)|\\
    && \ \ \ \ + \ \psi(\rho_t(x+y)) \cdot |2\zeta\phi_r'(\zeta^2)| \cdot |\rho_t(x+y) - \rho_t(x)|,
  \eeq
which implies by Condition (\ref{eqJumpMetric}) that
  \beq
    |h_r(x+y)-h_r(x)|
    &\leqslant&  \frac{\delta}{2} e^{\frac{\delta}{2} K} \cdot f_r(x) \cdot |\rho_t(x+y) - \rho_t(x)|\\
    && \ \ \ + \ \frac{3e^{\frac{\delta}{2} K}}{N}  \cdot f(x) \mathbf{1}_{r-K< \rho_t^2(x) < r+N+K}  \cdot |\rho_t(x+y) - \rho_t(x)|.
  \eeq
Due to $\Gamma_\nu(\rho_t,\rho_t) \leqslant 1$, it follows
  \beq
    \Gamma_\nu(h_r,h_r)
     &=& \frac12\int_{\mathbb{R}^m-\{0\}} |h_r(x+y)-h_r(x)|^2 \nu(x,\mathrm{d}y) \\
     &\leqslant& \frac12 \delta^2 e^{\delta K} f_r^2(x) + \frac{18e^{\delta K}}{N^2}  f^2(x) \mathbf{1}_{r-K < \rho_t^2(x)< r+N+K} \\
     &\leqslant& \frac12 \delta^2 e^{\delta K} f_r^2(x) + \frac{18e^{\delta (2N+3K)}}{N^2}  e^{\delta (r-N-K)} \mathbf{1}_{r-N-K < \rho_t^2(x)< r+N+K}.
  \eeq
Let $\eta_1 = \frac{\delta^2}{2c} e^{\delta K}$ and $\eta_2 = \frac{18e^{\delta (2N+3K)}}{N^2c}$, inserting the above estimate into (\ref{eqTrunct}) gives
  \beq
    \int f_r^2 \mathrm{d}\mu &\leqslant& \eta_1 \int f_r^2 \mathrm{d}\mu +
       \frac{b+ct}{c}\int h_r^2 \mathrm{d}\mu \\
    && \ +\ \eta_2 e^{\delta (r-N-K)} \mu\{ r-N-K < \rho_t^2< r+N+K \}.
  \eeq

{\bf Step 3}. Choose some big $N$ and small $\delta$ so that $\eta_1+ 2\eta_2 < 1$. Since $\mu$ is a probability, there exists a sequence of $n_k \in \mathbb{N}$ such that for each $r_k = n_k(N+K)$
  \[ \mu\{r_k - N -K < \rho_t^2 < r_k\} \geqslant \mu\{r_k < \rho_t^2 < r_k + N+K\},\]
which implies
  \beq
    && e^{\delta (r-N-K)} \mu\{r_k-N-K < \rho_t^2 < r_k+N+K \} \\
    &\leqslant& 2\int f^2 \mathbf{1}_{r_k-N-K < \rho_t^2\leqslant r_k} \mathrm{d}\mu
      \ \ \leqslant\ \ 2\int f_{r_k}^2  \mathrm{d}\mu.
   \eeq
It follows from Step 2
  \[ \int f_{r_k}^2 \mathrm{d}\mu \leqslant (\eta_1+2\eta_2) \int f_{r_k}^2 \mathrm{d}\mu + \frac{b+ct}{c} \int h_{r_k}^2 \mathrm{d}\mu, \]
and thus
  \[ \int f_{r_k}^2 \mathrm{d}\mu \leqslant \frac{b+ct}{c(1-\eta_1 - 2\eta_2)}\int h_{r_k}^2 \mathrm{d}\mu =: C\int h_{r_k}^2 \mathrm{d}\mu. \]

Recall $h_r=\frac{f_r}{\rho_t}$, fix a domain $\Omega$ with $\rho_t^2 \geqslant 2C$ on $\Omega^{\mathrm{c}}$, which means for $r_k> \mathrm{diam}\Omega$
  \[ \int f_{r_k}^2 \mathrm{d}\mu \leqslant C\int_{\Omega} \frac{f^2}{\rho_t^2} \mathrm{d}\mu + \frac12 \int_{\Omega^{\mathrm{c}}} f_{r_k}^2 \mathrm{d}\mu. \]
Consequently, we get $\int f^2 \mathrm{d}\mu = \lim\limits_{k\to \infty} \int f_{r_k}^2 \mathrm{d}\mu \leqslant 2C\int_{\Omega} \frac{f^2}{\rho_t^2} \mathrm{d}\mu < \infty$.
\end{proof}

\bigskip
\section{Proof of Proposition \ref{propGauIntGoz}}
\label{Proof2}
 \setcounter{equation}{0}

We firstly derive a Poincar\'{e} like inequality.

\begin{lemma} \label{lemGozPoin}
If the Gozlan's type condition (\ref{eqGozType}) holds, there exist two constants $\lambda_1$ and $\lambda_2$ with big $R$ such that for any $h\in \mathcal{D}(\mathcal{E})$
\[ \int h^2\mathrm{d}\mu \leqslant \lambda_1\int \sum\limits_{i=1}^m \frac{|h_i'|^2}{1+x_i^2}\mathrm{d}\mu +
       \lambda_2 \int_{B(0,R+1)} h^2\mathrm{d}\mu. \]
\end{lemma}
\begin{proof}
For convenience, denote $a=\frac{23}{27}$, $\mathrm{d}\nu_i=e^{-aV}\mathrm{d}x_i$ and
 \[ \mathrm{d}\hat{x}_i = \mathrm{d}x_1\cdots \mathrm{d}x_{i-1}\mathrm{d}x_{i+1}\cdots \mathrm{d}x_m\]
so that $\mathrm{d}\mu = e^{-(1-a)V}\mathrm{d}\nu_i \mathrm{d}\hat{x}_i$. Define $\phi_r\in C^1(\mathbb{R}^n)$ as
   \[ \phi_r(x) = \left\{
       \begin{array}{ll}
        1, & |x|\leqslant r;\\
        2(|x|-r)^3 - 3(|x|-r)^2 + 1, & r<|x|<r+1;\\
        0, & |x|\geqslant r+1,
       \end{array}
      \right.  \]
which satisfies $0\leqslant \phi_r \leqslant 1$ and $\left|(\phi_r)_i'\right| \leqslant 6\frac{|x_i|}{|x|}\sqrt{1-\phi_r}$. The proof has three steps.

{\bf Step 1}. For any $\varepsilon>0$, there exists $R>0$ by (\ref{eqGozType}) such that for all $|x|\geqslant R$
   \[  \sum\limits_{i=1}^m  \left(a|V_i'|^2 - V_{ii}''\right) \frac{1}{1+x_i^2} \geqslant m-\varepsilon. \]
It follows for any $h\in \mathcal{D}(\mathcal{E})$
  \beqn
    && (m-\varepsilon)\int h^2\mathrm{d}\mu \ =\ (m-\varepsilon)\int h^2\phi_R +  h^2(1-\phi_R)\mathrm{d}\mu \nonumber\\
    &\leqslant& (m-\varepsilon)\int h^2\phi_R\mathrm{d}\mu + \int h^2(1-\phi_R) \sum\limits_{i=1}^m \left(a|V_i'|^2 - V_{ii}''\right) \frac{1}{1+x_i^2} \mathrm{d}\mu \nonumber\\
    &=& (m-\varepsilon)\int h^2\phi_R\mathrm{d}\mu + \sum\limits_{i=1}^m \int \frac{h^2(1-\phi_R)e^{-(1-a)V}}{(1+x_i^2)} \left(a|V_i'|^2 - V_{ii}''\right)  \mathrm{d}\nu_i\mathrm{d}\hat{x}_i. \label{eqGoz00}
  \eeqn

Set $U^{(i)} = \frac{h^2(1-\phi_R)e^{-(1-a)V}}{1+x_i^2}$. For the reader's convenience, recall the integration by parts formula satisfied by $\nu_i$, we have
  \beq
    && \int U^{(i)} \left(a|V_i'|^2 - V_{ii}''\right) \mathrm{d}\nu_i\mathrm{d}\hat{x}_i
    \ =\ \int (U^{(i)})_i' V_i'  \mathrm{d}\nu_i\mathrm{d}\hat{x}_i\\
    &=& \int {\Big[} 2hh_i'V_i'(1-\phi_R)
    - (\phi_R)_i'h^2V_i' - \frac{2x_i}{1+x_i^2}h^2V_i'(1-\phi_R) \\
    && \hspace*{1.5cm}   - (1-a)h^2|V_i'|^2(1-\phi_R) {\Big]} \frac{1}{1+x_i^2} \mathrm{d}\mu.
  \eeq
Using the Cauchy-Schwarz inequality gives for any positive $\varepsilon_1,\varepsilon_2$ and $\varepsilon_3$
  \beq
    2hh_i'V_i' &\leqslant& \varepsilon_1h^2|V_i'|^2 + \varepsilon_1^{-1}|h_i'|^2,\\
    -(\phi_R)_i'h^2V_i' &\leqslant& 6\frac{|x_i|}{|x|}\sqrt{1-\phi_R}\cdot h^2|V_i'|\ \leqslant\ 3\varepsilon_2h^2|V_i'|^2(1-\phi_R) + 3\varepsilon_2^{-1}\frac{|x_i|^2}{|x|^2}h^2, \\
    -\frac{2x_ih^2V_i'}{1+x_i^2} &\leqslant& \varepsilon_3h^2|V_i'|^2 + \frac{x_i^2h^2}{\varepsilon_3(1+x_i^2)^2},
  \eeq
which implies by combining the above estimates subject to $\varepsilon_1 + 3\varepsilon_2 + \varepsilon_3 = 1-a$
  \beqn
    &&\int U^{(i)} \left(a|V_i'|^2 - V_{ii}''\right) \mathrm{d}\nu_i\mathrm{d}\hat{x}_i \nonumber\\
    &\leqslant& \int \frac{|h_i'|^2(1-\phi_R)}{\varepsilon_1(1+x_i^2)} + \frac{3|x_i|^2h^2}{\varepsilon_2(1+x_i^2)|x|^2} + \frac{x_i^2h^2(1-\phi_R)}{\varepsilon_3(1+x_i^2)^3}\mathrm{d}\mu. \label{eqGoz10}
  \eeqn

{\bf Step 2}. Since $\frac{x_i^2}{(1+x_i^2)^3}\leqslant \frac{4}{27}$ for any $x_i$ and there exists $x_j$ with $|x_j|^2\geqslant |x|^2/m$, we have
  \[ \sum\limits_{i=1}^m \frac{x_i^2}{(1+x_i^2)^3} \leqslant \frac{4}{27}(m-1) + \frac{1}{(1+m^{-1}|x|^2)^2}, \]
which implies
  \beqn
       \sum\limits_{i=1}^m \int \frac{x_i^2h^2(1-\phi_R)}{\varepsilon_3(1+x_i^2)^3}\mathrm{d}\mu \leqslant \left(\frac{4(m-1)}{27\varepsilon_3} + \frac{m^2}{\varepsilon_3R^4}\right) \int_{B(0,R)^c} h^2\mathrm{d}\mu.  \label{eqGoz12}
  \eeqn
We also have
  \beqn
      \sum\limits_{i=1}^m \int  \frac{3|x_i|^2h^2}{\varepsilon_2(1+x_i^2)|x|^2} \mathrm{d}\mu
       \leqslant  \frac{3}{\varepsilon_2}\int_{B(0,R)} h^2\mathrm{d}\mu + \frac{3m}{\varepsilon_2R^2} \int_{B(0,R)^c} h^2\mathrm{d}\mu.  \label{eqGoz15}
  \eeqn

Choose $R$ (depending on $\varepsilon$ and $\varepsilon_{1,2,3}$) so big that $\frac{m^2}{\varepsilon_3R^4} + \frac{3m}{\varepsilon_2R^2} \leqslant \varepsilon$, then combining (\ref{eqGoz00}-\ref{eqGoz15}) gives
  \beqn
     && (m-\varepsilon)\int h^2\mathrm{d}\mu  \ \leqslant\ \frac{1}{\varepsilon_1}\int \sum\limits_{i=1}^m \frac{|h_i'|^2}{1+x_i^2}\mathrm{d}\mu \ + \nonumber\\
     && \hspace*{1.8cm} \left( m-\varepsilon + \frac{3}{\varepsilon_2}\right) \int_{B(0,R+1)} h^2\mathrm{d}\mu  + \left(\frac{4(m-1)}{27\varepsilon_3} + \varepsilon \right) \int h^2\mathrm{d}\mu. \label{eqGoz20}
  \eeqn

{\bf Step 3}. We have to decide the range of $\varepsilon$ and $\varepsilon_{1,2,3}$. First of all, fix $\varepsilon_1 < \frac{4}{27m}$, and take any $\varepsilon_2$ such that $\varepsilon_1 + 3\varepsilon_2 < \frac{4}{27m}$ too. It follows
  \[ \frac{4(m-1)}{27\varepsilon_3} = \frac{4(m-1)}{27(1-a-\varepsilon_1-3\varepsilon_2)} < m, \]
so we can take any $\varepsilon$ such that $\frac{4(m-1)}{27\varepsilon_3} + 2\varepsilon < m$.

Now, using (\ref{eqGoz20}) yields
  \beqn
    \int h^2\mathrm{d}\mu \leqslant \lambda_1\int \sum\limits_{i=1}^m \frac{|h_i'|^2}{1+x_i^2}\mathrm{d}\mu +
       \lambda_2 \int_{B(0,R+1)} h^2\mathrm{d}\mu, \label{eqGozPoincare}
  \eeqn
where $\lambda_1 = [\varepsilon_1(m-2\varepsilon - \frac{4(m-1)}{27\varepsilon_3})]^{-1}$ and $\lambda_2 = \left(m-\varepsilon + 3\varepsilon_2^{-1}\right)(m-2\varepsilon - \frac{4(m-1)}{27\varepsilon_3})^{-1}$. The proof is completed.
\end{proof}

Now, we prove Proposition \ref{propGauIntGoz}.
\begin{proof}
Let $\beta_n = \int |x|^{2n} \mathrm{d}\mu$. Applying (\ref{eqGozPoincare}) to $h(x) = |x|^n$ yields
  \beqn
   \beta_n
    &\leqslant& \lambda_1 \int \sum\limits_{i=1}^m \frac{n^2x_i^2}{1+x_i^2} |x|^{2n-4} \mathrm{d}\mu + \lambda_2 \int_{B(0,R+1)} |x|^{2n} \mathrm{d}\mu \nonumber\\
    &\leqslant& \lambda_1mn^2 \int |x|^{2n-4} \mathrm{d}\mu + \lambda_2(R+1)^2 \int_{B(0,R+1)} |x|^{2n-2} \mathrm{d}\mu \nonumber\\
    &\leqslant& \lambda_1mn^2 \beta_{n-2} + \lambda_2(R+1)^2 \beta_{n-1}, \label{eqRecurGoz}
  \eeqn
which implies all $\beta_n < \infty$.

For simplicity, abbreviate ${\lambda'}_1 = \lambda_1m$ and ${\lambda'}_2 = \lambda_2(R+1)^2$. Combining the H\"{o}lder inequality with (\ref{eqRecurGoz}) gives
  \[ \beta_n = \int |x|^{n+1}|x|^{n-1} \mathrm{d}\mu \leqslant \beta_{n+1}^{\frac12}\beta_{n-1}^{\frac12}
     \leqslant \left[{\lambda'}_1(n+1)^2\beta_{n-1} + {\lambda'}_2\beta_n \right]^{\frac12} \beta_{n-1}^{\frac12}, \]
which implies
  \[ \beta_n \leqslant \frac{{\lambda'}_2 + \sqrt{{\lambda'}_2^2 + 4{\lambda'}_1(n+1)^2}}{2} \beta_{n-1}  \leqslant \left[{\lambda'}_2 + \sqrt{{\lambda'}_1}(n+1) \right] \beta_{n-1}. \]
Choose any $\gamma>\sqrt{{\lambda'}_1}$, it follows ${\lambda'}_2 + \sqrt{{\lambda'}_1}(n+1) \leqslant \gamma n$ for big $n$, which yields a constant $C$ such that for all $n$
  \[ \beta_n \leqslant C\gamma^n n!. \]
By the same argument as (\ref{eqExpInt}) for any $\delta < \gamma^{-1} < {\lambda'}_1^{-\frac12}$, we have $\mu e^{\delta |x|^2}< \infty$.

Recall the constraints on all parameters (See Step 3 in the proof of Lemma \ref{lemGozPoin}), $\delta$ is allowed to be not greater than
  \[ \sup\left\{{\lambda'}_1^{-\frac12} :\ \varepsilon_1+3\varepsilon_2+\varepsilon_3 = 1-a, \ \varepsilon_1<\frac{4}{27m}, \ \varepsilon=\varepsilon_2=0 \right\} ,\]
which achieves $\frac{2(\sqrt{m} - \sqrt{m-1})}{3\sqrt{3m}}$. The proof is completed.
\end{proof}

\bigskip

\subsection*{Acknowledgements}

{\small I deeply appreciate two anonymous referees for their conscientious reading and inspiring suggestions on the first version that very much helped improve this paper. I also thank the financial supports from NSFC (no. 11201456, no. 1143000182), AMSS research grant (no. Y129161ZZ1), and Key Laboratory of Random Complex Structures and Data, Academy of Mathematics and Systems Science, Chinese Academy of Sciences (No. 2008DP173182).}




%

\end{document}